\documentclass[11pt,leqno]{amsart}
\usepackage{amscd}
\usepackage{amssymb}
\usepackage{amsfonts}
\usepackage{latexsym}
\usepackage{verbatim}

\theoremstyle{plain}
\newtheorem{theorem}{Theorem}[section]
\newtheorem{definition}[theorem]{Definition}
\newtheorem{lemma}[theorem]{Lemma}
\newtheorem{prop}[theorem]{Proposition}
\newtheorem{cor}[theorem]{Corollary}
\newtheorem{rem}[theorem]{Remark}

\begin{document}
\title{Moduli Space of CR-Projective Complex Foliated Tori}
\author[Costantino Medori \and Adriano Tomassini]{Costantino Medori \and Adriano Tomassini$^*$}
\thanks{$^*$ corresponding author}
\date{}
\address{Dipartimento di Matematica\\ Universit\`a di Parma\\ Viale G.P. Usberti 53/A\\
43100 Parma\\ Italy}
\email{costantino.medori@unipr.it}
\address{Dipartimento di Matematica\\ Universit\`a di Parma\\ Viale G.P. Usberti 53/A\\
43100 Parma\\ Italy}
\email{adriano.tomassini@unipr.it}
\thanks{This work was supported by the Project MIUR ``Geometric Properties of Real and Complex Manifolds'' and by GNSAGA
of INdAM}
\begin{abstract} We study the moduli space of CR-projective complex foliated tori. We describe it in terms of isotropic subspaces
of Grassmannian and we show that it is a normal complex analytic
space.
\bigskip\par\noindent
\textsc {R\'{e}sum\'{e}.} Nous \'{e}tudions l'\'{e}space des
modules des tores feuillet\'{e}s par des feuilles complexes,
plong\'{e}s de fac\c{o}n CR dans un \'{e}space projective. Nous le d\'{e}crivons en
termes des sous-\'{e}spaces isotropiques de la Grassmanni\`{e}nne
et nous prouvons que c'est un \'{e}space analytique complexe
normale.
\end{abstract}
\maketitle
\newcommand\C{{\mathbb C}}
\newcommand\N{{\mathbb N}}
\newcommand\Proj{{\mathbb P}}
\newcommand\R{{\mathbb R}}
\newcommand\Z{{\mathbb Z}}
\newcommand\T{{\mathbb T}}
\newcommand{\de}[2]{\frac{\partial #1}{\partial #2}}
\section{Introduction}
A {\em complex foliated torus} or {\em CR-torus} is defined as
$\mathbb{T}^{n,k}=\C^n\times \R^k/\Gamma$, where
$\Gamma\subset\C^n\times\R^k$ is a lattice. The natural projection
of $\C^n\times\R^k$ on $\R^k$ induces a foliation of
$\mathbb{T}^{n,k}$ whose leaves are complex manifolds of dimension
$n$. Hence, $\mathbb{T}^{n,k}$ is a compact Levi-flat manifold.
\newline Complex foliated tori are related to {\em quasi-Abelian
varieties}, i.e. to quotients of $\C^n$ by a discrete subgroup. In
fact, any quasi-Abelian variety, as real manifold, is
diffeomorphic to the product of a CR-torus and a real linear
vector space (see e.g. \cite{AG1}, \cite{AG2}). For more details
and results on quasi-Abelian varieties we refer to \cite{AK},
\cite{CC}. \newline In this paper we consider complex foliated
tori endowed with a {\em polarization} $\omega$, namely $\omega$
is a bilinear skew-symmetric form on $\R^{2n}\times\R^{k}$ taking
integral values on $\Gamma\times\Gamma$ and such that its
restriction to $\R^{2n}$ is the imaginary part of a positive
definite Hermitian form on $\C^n$. The existence of a
polarization on a complex foliated torus $\mathbb{T}^{n,k}$ is
equivalent to the CR-embeddability
into a projective space, i.e. to the existence of an analytic embedding of $\mathbb{T}^{n,k}$
into a projective space that is holomorphic along the leaves (see \cite{CT} and the references included).
We will call {\em CR-projective} such tori. Note that the definition of
polarization we take in consideration here is more general of that
one considered in \cite{CT}.\newline
We define the {\em moduli space}
of CR-projective tori and we show that it is a normal complex
analytic space (see Theorem \ref{principale}). In order to do
this, we will present the space of period matrices as a subset of
the Grassmannian of $n$-complex planes in $\C^{2n+k}$.\smallskip\par\noindent
The authors would like to dedicate this paper to the memory of Nicolangelo Medori.
\section{Preliminaries}
Let $\C^n\times \R^k$ be endowed with the natural CR-structure. A {\em complex foliated torus} or
{\em CR-torus} is a torus
$$
\mathbb{T}^{n,k}=\C^n\times \R^k/\Gamma\,,
$$
where $\Gamma$ is a lattice in
$\R^{2n+k}=\R^{2n}\times\R^k$, with the natural CR-structure induced by $\C^n\times \R^k$.
Note that the projection $\pi :\C^n\times \R^k\to \mathbb{T}^{n,k}$ gives rise to a foliation on
$\mathbb{T}^{n,k}$, whose leaves are complex manifolds of dimension $n$.\newline
Let $(z,t)$ denote any point of $\C^n\times \R^k$. A {\em CR-map} $\varphi :\C^n\times \R^k\to \C^n\times \R^k$
is a smooth map of the form
$$
\varphi(z,t)=(f(z,t),h(t))
$$
where $f$ is holomorphic with respect to $z$. A CR-map between complex foliated tori
$\mathbb{T}^{n,k}=\C^n\times \R^k/\Gamma$ and
$\mathbb{T'}^{n,k}=\C^n\times \R^k/\Gamma'$ is given by a CR-map
$$
\tilde{\varphi} :\C^n\times \R^k\to \C^n\times \R^k
$$
such that $\tilde{\varphi}(\Gamma)\subset\Gamma'$. The following lemma characterizes the CR-maps between complex foliated tori.
\begin{lemma}\label{CRmaptori}
Let $\mathbb{T}^{n,k}=\C^n\times \R^k/\Gamma$ and $\mathbb{T'}^{n,k}=\C^n\times \R^k/\Gamma'$ be complex foliated tori and
$\varphi :\mathbb{T}^{n,k}\to \mathbb{T'}^{n,k}$ be a CR-map. Then $\varphi$ is induced by a CR-map
$\tilde{\varphi} : \C^n\times \R^k\to\C^n\times \R^k$ given by
$$
\tilde{\varphi}(z,t) =(Az+Bt+\beta(t),Ct+\gamma(t))\,,
$$
where $A\in M_{n,n}(\C)$, $B\in M_{n,k}(\C)$, $C\in M_{k,k}(\R)$ and
$$
\beta :\C^n\times \R^k\to \C^n\,,\quad \gamma :\C^n\times \R^k\to \R^k
$$
do not depend on $z$, are $\Gamma$-periodic and $(\beta(0),\gamma(0))$ belongs to $\Gamma'$.
\end{lemma}
\begin{proof}
As already remarked, $\varphi$ is induced by $\tilde{\varphi} : \C^n\times \R^k\to\C^n\times \R^k$,
$$
\tilde\varphi(z,t)=(f(z,t),h(t))\,.
$$
such that, $f(z,t)$ is a holomorphic map of $z$, for any $t\in \R^k$. As for the case of complex tori
(see e.g. \cite[Chap.1, Proposition\,2.1]{BL}) we get
\begin{equation}\label{CRmap1}
f(z,t)= A(t)z+B(t)\,,
\end{equation}
where $A(t)$ and $B(t)$ are a $n\times n$ complex matrix and a vector of $\C^n$ respectively, both depending on $t$. Since
\begin{equation}\label{CRmap2}
f(z+\zeta,t+\tau)-f(z,t)\in\Gamma'\,,
\end{equation}
for any $(z,t)\in \C^n\times \R^k\,$, $(\zeta,\tau)\in\Gamma\subset\C^n\times \R^k$, we have:
\begin{equation}\label{CRmapderivative1}
\frac{\partial}{\partial t_s}(f(z+\zeta,t+\tau)-f(z,t))=0\,,\qquad s=1,\ldots , k\,.
\end{equation}
By taking into account \eqref{CRmap1}, then \eqref{CRmapderivative1} implies that
$$
\frac{\partial}{\partial t_s}[(A(t+\tau)-A(t))z+A(t+\tau)\zeta+B(t+\tau)-B(t)]=0\,.
$$
By a CR-change of coordinates in $\C^n\times \R^k$, we may assume that the
lattice $\Gamma$ contains points of the form $(0,\tau)$.\newline
Hence, evaluating the last expression for $z=0$ and $\zeta=0$ , we
obtain
\begin{equation}\label{CRmapderivative2}
\frac{\partial}{\partial t_s}[B(t+\tau)-B(t)]=0\,.
\end{equation}
Hence
\begin{equation}\label{CRmapderivative3}
\frac{\partial}{\partial t_s}[A(t+\tau)(z+\zeta)-A(t)z]=0\,.
\end{equation}
Then \eqref{CRmapderivative3} and \eqref{CRmapderivative2} imply that
\begin{equation}\label{A}
A^{ij}(t)=\sum_{h=1}^kA^{ij}_ht_h+\alpha^{ij}(t)\,.
\end{equation}
and
\begin{equation}\label{B}
B(t)=Bt+\beta(t)\,,
\end{equation}
where $A^{ij}_h$, $B=(B^{ih})$ are constant and $\alpha^{ij}$, $\beta$ are periodic. Now we are going to show that
$$
A^{ij}_h=0\,,\quad h=1,\ldots ,k,\quad\hbox{\rm and $\alpha^{ij}\,,\quad i,j=1,\ldots, n$, are constant.}
$$
Differentiating \eqref{CRmap2} with respect to $z$ and taking onto account \eqref{CRmap1}, we obtain
$$
\begin{array}{lll}
 0& = & \frac{\partial}{\partial z_r}[f(z+\zeta,t+\tau)-f(z,t)]=\\[5pt]
 {}&=&\frac{\partial}{\partial z_r}[(A(t+\tau)-A(t))z+A(t+\tau)\zeta+B(t+\tau)-B(t)]\,.
\end{array}
$$
Therefore \eqref{A} and \eqref{B} imply that
$$
\frac{\partial}{\partial z_r}[\sum_{h=1}^{k}(\sum_{j=1}^n A^{ij}_h\tau_hz_j+B^{ih}\tau_h)]=0\,.
$$
Hence $A^{ij}_h=0$. By \eqref{CRmapderivative3} we get
$$
\frac{\partial}{\partial t_s}[A(t)\zeta]=0\,.
$$
Therefore $A$ is constant and the proposition is proved.
\end{proof}
Let $\mathbb{T}^{n,k}=\C^n\times \R^k/\Gamma$ be a complex foliated torus. Let
$$
\gamma_1=(z_1,t_1),\,\ldots , \gamma_{2n+k}=(z_{2n+k},t_{2n+k})\,,
$$
$z_1,\ldots ,z_{2n+k}\in\C^n\,,t_1,\ldots ,t_{2n+k}\in\R^k\,$, be a $\Z$-basis of $\Gamma$. Then
$$
\Omega=
\left(
\begin{array}{lll}
z_1 & \cdots & z_{2n+k}\\[5pt]
t_1 & \cdots & t_{2n+k}
\end{array}
\right)\,\in M_{n+k,2n+k}(\C)
$$
is called a {\em period matrix} of $\mathbb{T}^{n,k}$. Observe that $\Omega$ depends on the choice of
the $\Z$-basis of $\Gamma$. We have the following
\begin{prop}
Let $\mathbb{T}^{n,k}=\C^n\times \R^k/\Gamma$ and $\mathbb{T'}^{n,k}=\C^n\times \R^k/\Gamma'$ be
complex foliated tori. Denote by $\Omega$ and $\Omega'$ period matrices for $\mathbb{T}^{n,k}$ and $\mathbb{T'}^{n,k}$
respectively. Then $\mathbb{T}^{n,k}$ and $\mathbb{T'}^{n,k}$ are CR-isomorphic if and only if there exists
$$
M=
\left(
\begin{array}{ll}
A & B\\[5pt]
0 & C
\end{array}
\right),
$$
with $A\in\hbox{\rm GL}(n,\C)\,,B\in\hbox{\rm M}_{n,k}(\C)\,,C\in\hbox{\rm GL}(k,\R)$, and
$P\in\hbox{\rm GL}(2n+k,\Z)$ such that
\begin{equation}\label{toriequivalence}
M\Omega=\Omega' P\,.
\end{equation}
\end{prop}
\begin{proof}
If there exist $M$ and $P$ as in \eqref{toriequivalence}, then the map $\tilde\varphi(z,t)=(Az+Bt,Ct)$ induces
a CR-diffeomorphism between $\mathbb{T}^{n,k}$ and $\mathbb{T'}^{n,k}$. Vice versa, if $\varphi: \mathbb{T}^{n,k}\to \mathbb{T'}^{n,k}$
is a diffeomorphism, then Lemma \ref{CRmaptori} implies that $\tilde\varphi(z,t)=(Az+Bt+\beta(t),Ct+\gamma(t))$. Since
$(\beta(0),\gamma(0))\in\Gamma'$, then \eqref{toriequivalence} holds.
\end{proof}
\medskip\noindent
We will denote by $L_{n,k}$ the group of matrices given by
$$
L_{n,k}=\left\{
M=
\begin{pmatrix}
A & B\\
0 & C
\end{pmatrix}\,\Big\vert \, A\in\hbox{\rm GL}(n,\C)\,,B\in\hbox{\rm M}_{n,k}(\C)\,,C\in\hbox{\rm GL}(k,\R)
\right\}\,.
$$
Two period matrices $\Omega$ and $\Omega'$ are said to be {\em equivalent} if they satisfy
condition \eqref{toriequivalence} for some $M\in L_{n,k}$, $P\in\hbox{\rm GL}(2n+k,\Z)$. \newline
By definition, the {\em moduli space}
of {\em complex foliated tori} is the quotient $\mathcal{M}_{n,k}$ of the space of period matrices modulo the
equivalence relation \eqref{toriequivalence}.\newline
It can be checked that any equivalence class of period matrices has a representative $\Omega$ of the following form
$$
\Omega=
\begin{pmatrix}
Z & 0\\
T & I_k
\end{pmatrix}\,.
$$
We will call {\em adapted} such a period matrix.
\section{Moduli of polarized complex foliated tori}
Let $\mathbb{T}^{n,k}=\C^n\times \R^k/\Gamma$ be a complex foliated torus.
\begin{definition}\label{polarization}
A {\em polarization} on $\mathbb{T}^{n,k}$ is given by a skew-symmetric bilinear form
$\omega$ on $\R^{2n}\times\R^k$ such that:
\begin{enumerate}
\item[i)] $\omega_{\vert_{\Gamma\times\Gamma}}$ takes integral values;
\item[ii)] $\omega_{\vert_{\R^{2n}\times\{0\}}}$ is the imaginary part of a positive definite Hermitian form on $\C^n$.
\end{enumerate}
\end{definition}
A complex foliated torus $\mathbb{T}^{n,k}$ is said to be {\em CR-projective} if it can be CR-embedded into a $\Proj^N(\C)$.
It turns out that a complex foliated torus $\mathbb{T}^{n,k}$ is CR-projective if and only if
there exists a polarization $\omega$ on $\mathbb{T}^{n,k}$ (see \cite{CT}). For a CR-embedding theorem of a compact Levi-flat
manifold of codimension one can see \cite[Theorem 3]{OS}.

A pair $(\mathbb{T}^{n,k},\omega)$ is said to be a {\em polarized complex foliated torus}. Two
polarized complex foliated tori $(\mathbb{T}^{n,k},\omega)$ and $(\mathbb{T'}^{n,k},\omega')$ are said to be {\em equivalent}
if there exists a CR-diffeomorphism $\varphi : \mathbb{T}^{n,k}\to\mathbb{T'}^{n,k}$ such that
the restrictions of
$\omega$ and $\varphi^*\omega'$ to $\C^n\times\{0\}$ coincide.
\begin{rem}\label{adaptedform} Any polarized complex foliated torus $(\mathbb{T'}^{n,k},\omega')$ is
equivalent to $(\mathbb{T}^{n,k},\omega)$, where $\omega$ is represented by
$$
\begin{pmatrix}
\eta & 0\\
0 & 0
\end{pmatrix}
$$
and $\mathbb{T}^{n,k}$ has an adapted period matrix. To show this, set
$\eta = \omega'_{\vert_{\C^n\times\{0\}}}$. Let
$$
\Omega' =
\begin{pmatrix}
Z' & W'\\
T' & R'
\end{pmatrix}
$$
be a period matrix for $\mathbb{T'}^{n,k}$. By changing the order of the columns of $\Omega'$ we
may assume that $R'$ is invertible. Then, by acting on the left with
$$
M=
\begin{pmatrix}
A & -AW'R'^{-1}\\
0 & R'^{-1}
\end{pmatrix}\,,
$$
where $A\in {\rm U}(n)$, we obtain that $\Omega$ is an adapted period matrix and that the restriction of
$\varphi^*\omega'$ to $\C^n\times\{0\}$
is still $\eta$, since $A\in {\rm U}(n)$ (where $\varphi$ is the CR-diffeomorphism represented by the matrix $M$).\newline
As for the complex case, there exist $d_1,\ldots ,d_n\in\N$ with $d_{i}\vert d_{i+1}\,,i=1,\ldots ,n-1$,
such that the polarization $\omega$ can be represented by the matrix
\begin{equation}\label{canonicalform}
\begin{pmatrix}
0 & \Delta & 0\\
-\Delta &0 &0\\
0 & 0 & 0
\end{pmatrix}\,,
\end{equation}
where $\Delta$ is the diagonal matrix having diagonal entries $d_1,\ldots ,d_n$ (see e.g. \cite[Chap.VI, Proposition\,1.1]{D},
\cite[Lemma p.204]{GH}).
\end{rem}
Let
\begin{eqnarray*}
F_{n,k}&=&\left\{
M=
\begin{pmatrix}
A & 0\\
0 & I_k
\end{pmatrix}\,\Big\vert\,\ A\in {\rm U}(n)
\right\}\,,\\
H_{n,k}&=&\left\{
P=
\begin{pmatrix}
\alpha & 0\\
\beta & I_k
\end{pmatrix}\,\Big\vert\,\ \alpha\in {\rm Sp}(2n,\Z)\,,\beta\in {\rm M}_{k,2n}(\Z)\,
\right\}\,.
\end{eqnarray*}
Let $\mathbb{T}^{n,k}$ and $\mathbb{T'}^{n,k}$ be two complex foliated tori endowed with the same polarization
$\omega$. As observed in Remark \ref{adaptedform}, we may assume that $\omega$ has the canonical form
given by \eqref{canonicalform} and $\Omega$ and $\Omega'$ are adapted period matrices
of $\mathbb{T}^{n,k}$ and $\mathbb{T'}^{n,k}$ respectively.\newline
We say that $\Omega$ and $\Omega'$ are {\em equivalent} if there exist $M\in F_{n,k}$ and $P\in H_{n,k}$ such that
\begin{equation}\label{polarizedtoriequivalence}
M\Omega = \Omega' P\,.
\end{equation}
\begin{definition}
The {\em moduli space} of CR-projective complex foliated tori is the quotient $\mathcal{M}_{n,k}^\omega$ of the space
of period matrices modulo the
equivalence relation \eqref{polarizedtoriequivalence}.
\end{definition}
We are going to describe explicitly the space of period matrices. Let
$\Omega$ be an adapted period matrix for the polarized complex foliated
torus $(\mathbb{T}^{n,k},\omega)$.
Denote by $\gamma_1,\ldots,\gamma_{2n+k}$ the column vectors of $\Omega$. Set
$$
\left\{
\begin{array}{lll}
J\gamma_i &= \gamma_{n+i}\,,\quad &i=1,\ldots ,n\,,\\
J\gamma_{n+i} &= -\gamma_i \,,\quad &i=1,\ldots ,n\,.
\end{array}
\right.
$$
Then $J$ defines a complex structure on the real vector space $V$ spanned by
$\gamma_1,\ldots , \gamma_{2n}$. These
vectors give an $\R$-isomorphism $f$ between $V$ and $\R^{2n}$ by setting
$f(\gamma_i) = e_i\,,i=1,\ldots ,2n$. In this way $(V,J)$ is isomorphic to $(\R^{2n},J_0)$. Then
$$
g(u,v)=f^*\omega(u,Jv)
$$
is a $J$-invariant inner product on $V$. Then, by setting
$$
L=V^{1,0}=\left\{x-iJx\,\vert\,x\in V\right\}
\,,
$$
we get that the complex $n$-plane $L\subset\C^{2n+k}$ satisfies
$$
L\cap\overline{L}=\{0\}\,,\quad  f^*(\hat\omega_{\vert_{\C^{2n}}})=0\,,
$$
where $\hat\omega$ is the complexification of $\omega$.
Hence, we have proved the following
\begin{prop}\label{period}
The space of period matrices of CR-projective complex foliated tori is given by
$$
\mathcal{U}_{n,k}^\omega=\left\{L\in \hbox{\rm Gr}_\C(n,2n+k)\,\vert \, L\cap\overline{L}=\{0\}\,,\,\, L\,\,
\hbox{\rm is}\,\, f^*\hat\omega\hbox{\rm -isotropic}\right\}\,.
$$
\end{prop}
As a consequence of the previous proposition we get the following
\begin{cor}
The space of period matrices of CR-projective complex foliated tori is a K\"ahler manifold of
complex dimension $\frac{1}{2}n(n+1+2k)$.
\end{cor}
Now we can describe the moduli space (for the case of Abelian varieties see e.g. \cite{BL}, \cite{D}). We have the
following
\begin{theorem}\label{principale}
The moduli space $\mathcal{M}_{n,k}^\omega$ of CR-projective complex foliated tori is given by
$$
\mathcal{M}_{n,k}^\omega\simeq\mathcal{U}_{n,k}^\omega\big/H_{n,k}\,.
$$
Any discrete subgroup of $H_{n,k}$ acts properly discontinuously on
$\mathcal{U}_{n,k}^\omega$. In particular, $\mathcal{M}_{n,k}^\omega$ is a normal analytic space.
\end{theorem}
\begin{proof} By Proposition \ref{period}, we immediately get that $\mathcal{M}_{n,k}^\omega\simeq \mathcal{U}_{n,k}^\omega\big/H_{n,k}\,$.
\newline
Now, we show that any discrete subgroup $G$ of $H_{n,k}$ acts properly discontinuously on
$\mathcal{U}_{n,k}^\omega$. Consider two compact sets $K_1$ and $K_2$ in $\mathcal{U}_{n,k}^\omega$. Let
$$
M=
\begin{pmatrix}
\alpha & 0\\
\beta & I_k
\end{pmatrix}
\in G
$$
such that $MK_1\cap K_2\neq\emptyset$. Let $V_M\in MK_1\cap K_2$ and set $V'_M = M^{-1} V_M$. Denote
by $J$ and $J'$ be
the corresponding complex structures on $V_M$ and $V'_M$, respectively. Then we have
$$
\omega(Ju,Jv)=\omega(u,v)\,,\forall u,v\in V_M\,,\,\,\omega(u,Ju)>0\,,\,\forall u\neq 0
$$
and the same holds for $V'_M$. Take a symplectic basis
${\mathcal B}=\{\gamma_1,\ldots ,\gamma_{2n}\}$ of $V_M$ and let
${\mathcal B}'=\{\gamma'_1=M^{-1}\gamma_1,\ldots ,\gamma'_{2n}=M^{-1}\gamma_{2n}\}$. Then, with respect to the
bases ${\mathcal B}$ and ${\mathcal B}'$ of $V_M$
and $V'_M$ respectively (denoting with the same letters the matrices for $J$, $J'$ and
$\omega$), we have $J'=\alpha J\alpha^{-1}$ and, consequently, ${}^tJ\omega J=\omega\,,\,\,{}^tJ'\omega J'=\omega$. Therefore,
${}^tJ\omega$ and ${}^tJ'\omega$ are symmetric and positive definite matrices. Hence,
there exist two orthogonal matrices $Q$, $Q'$ such that
$$
{}^tJ\omega ={}^tQ D Q\,,\quad{}^tJ'\omega ={}^tQ' D' Q'\,,
$$
where $D$ and $D'$ are diagonal and positive definite matrices. By taking into account the previous relations, we get
$D'={}^t SDS$, where $S=Q\alpha Q'$, so that $S$ varies in a compact set. Therefore, $\alpha$ lies in a
compact set. By the definition of the action, it can be easily checked that also $\beta$ varies in a compact set. Hence $G$ is finite.
\end{proof}

\end{document}